\UseRawInputEncoding
\documentclass[11pt]{amsart}
\usepackage{amsfonts,amsmath,amssymb,amsthm}
\usepackage{color}
\usepackage[hypertex,colorlinks,backref]{hyperref}
\usepackage{cite}

\newtheorem{theorem}{Theorem}
\theoremstyle{plain}

\newtheorem{corollary}{Corollary}

\newtheorem{proposition}{Proposition}

\numberwithin{equation}{section}

\begin{document}
\title{Generalized Quasi Yamabe Gradient Solitons and Applications}
\author{S\.INEM G\"uler}
\address[S. G\"uler]{Department of Industrial Engineering,
Istanbul Sabahattin Zaim University, Halkal\.i, Istanbul, Turkey.}
\email{sinem.guler@izu.edu.tr}
\author{B\"ulent \"Unal}
\address[B. \"{U}nal]{Department of Mathematics, Bilkent University,
Bilkent, 06800 Ankara, Turkey}
\email{bulentunal@mail.com}

\begin{abstract}
The purpose of this article is to study generalized quasi Yamabe gradient
solitons  on warped product manifolds. First, we obtain some necessary and
sufficient conditions for the existence of  generalized quasi Yamabe gradient
solitons  equipped with the warped product structure. Then we study three
important applications in the Lorentzian and the neutral  settings for the
particular class, called as gradient Yamabe soliton: We proved the existence
of the non-trivial gradient Yamabe soliton on generalized Robertson-Walker
spacetimes, standard static spacetimes and Walker manifolds.
\end{abstract}

\maketitle

\section{Introduction}

In recent years, self-similar solutions, referred as soliton solutions of some geometric flow equations, have begun to be introduced and
studied as they appear to be possible singularity models. The most famous and studied classes of them are the Ricci solitons, defined as fixed
points of the Ricci flow and in this area, significant progress has been made. Then, in the late 1980's, Hamilton introduced the Yamabe flow to
prove the Yamabe problem, \cite{Hamilton88}. Basically, the Yamabe problem consists of looking for a metric on an $n\geq 3$ dimensional
manifold such that its scalar curvature is constant. Thus, the Yamabe flow is defined as the metric $g(t)$ on a Riemannian manifold $(M^n,g)$
satisfying $\frac{\partial g(t)}{\partial t}=-\tau g(t)$, where $\tau$ is the scalar curvature of $M$. The solution of this problem for
two-dimensional case is already guaranteed by the Uniformization Theorem.   For the studies on this concept we refer to \cite{Brendle1,
Brendle2, Daskalopoulos}. Gradient Yamabe solitons are the solutions of this flow and defined as follows:

A pseudo Riemannian manifold $(M^n,g)$ is said to be a gradient Yamabe soliton if there exists a smooth function $\varphi$ on $M$ and a
constant $\lambda$ satisfying
\begin{equation} \label{meqn:gradient-Yamabe}
{\rm Hess}(\varphi)=(\tau-\lambda)g.
\end{equation}
If $\lambda > 0, \lambda < 0$ or $\lambda = 0$, then $(M^n,g)$ is called a  shrinking, expanding or steady gradient Yamabe soliton
respectively.  One of the most important results recorded as the solution to the Yamabe problem is given in \cite{Daskalopoulos} and it is
proved  that the metric of any compact Yamabe gradient soliton is a metric of constant scalar curvature.

In \cite{Cao2012} it was proved that gradient Yamabe soliton admits a warped product structure, and this result enables to make possible
studies in both Riemannian and Lorentzian settings.

After  the Ricci flow theory is introduced and some progess was made, to classify the Riemannian manifolds and  generalize the Ricci solitons,
almost gradient Ricci solitons, quasi Einstein manifolds and generalized quasi Einstein manifolds have been introduced and studied extensively.
For these concepts, we refer to \cite{Cao-Ricci,Case2010,Case2011,BejanGuler,Catino,AltayGuler} and many others. Analogously, in literature
some generalizations of the self-similar solutions of Yamabe flow are also described. First the notion of quasi Yamabe gradient soliton in
\cite{Huang,Neto2015} and then  notion of generalized quasi Yamabe gradient soliton were introduced in \cite{Neto2016}:

An $n$-dimensional pseudo-Riemannian manifold $(M,g)$ is said
to be a generalized quasi Yamabe gradient soliton if there
exist smooth functions $\varphi$ and $\mu$ on $M$ and also a
constant $\lambda$ satisfying
\begin{equation} \label{meqn1}
{\rm Hess}(\varphi)=(\tau-\lambda)g + \mu {\rm d} \varphi \otimes {\rm d} \varphi
\end{equation}
\noindent where ${\rm d} \varphi$ is the dual 1-form of $\nabla \varphi$ and
$\tau$ is the scalar curvature of $M$. Here, $\varphi$ is called as
a potential function and the underlying generalized quasi Yamabe gradient soliton
(briefly GQY) is denoted by $(M,g,\varphi,\mu,\lambda).$

If  $\varphi$ is a constant function, we say that  $(M,g)$ is a trivial generalized quasi Yamabe gradient soliton. Otherwise, it will be called
non-trivial. The restricted case where $\mu$ is constant is called quasi Yamabe gradient soliton. Moreover, if  $\mu=0,$ then the equation
$\eqref{meqn1}$ reduces to the fundamental equation of  gradient Yamabe soliton \eqref{meqn:gradient-Yamabe}.

In \cite{Huang} it is proved that compact quasi Yamabe gradient soliton has constant scalar curvature and in \cite{Wang}, Wang has studied the
special generalized quasi Yamabe gradient soliton in which $\mu =\frac{1}{m}$  for some constant $m>0$ and proved that $m$-quasi Yamabe
gradient soliton also has warped product structure  in the region $||\nabla \varphi || \neq 0$. Moreover, the warping function is completely
defined by the potential function of the soliton.
Then in \cite{Neto2016}, Neto and Oliveira have extended  these results to the  generalized quasi Yamabe gradient solitons.

Inspired by these studies, we investigate some necessary and sufficient conditions for the existence of  generalized quasi Yamabe gradient
solitons  equipped with the warped product structure.  Then we study three important applications in the Lorentzian and the neutral  settings
for the class of  gradient Yamabe solitons.  We proved the existence of  the non-trivial gradient Yamabe soliton on generalized
Robertson-Walker spacetimes,  standard static spacetimes and Walker manifolds.

%\section{Preliminaries}

\section{Warped Product Generalized Quasi Yamabe Gradient Solitons}

Assume that $(B,g_B)$ and $(F,g_F)$ are two pseudo-Riemannian manifolds
of dimensions $r$ and $s,$ respectively. Let
$\pi:B \times F \rightarrow B$ and $\sigma: B \times F \rightarrow F$ be
the natural projection maps of the Cartesian product $B \times F$ onto $B$ and $F,$
respectively. Also, let
$b:B \rightarrow \left( 0,\infty \right) $ be a positive real-valued
smooth function. The warped product manifold $M=B\times _{b} F$ is the
the product manifold $B \times F$ equipped with the metric tensor defined by
\begin{equation}
\label{warped metric}
g=\pi^{\ast }\left( g_{B}\right) \oplus \left( b \circ \pi \right)
^{2}\sigma^{\ast }\left( g_{F}\right)
\end{equation}%
where $^{\ast }$ denotes the pull-back operator on tensors\cite%
{Bishop:1969,Oneill:1983,Shenawy:2015}. The function $b$ is called the
warping function of the warped product manifold $B \times _{b} F$, and the manifolds $B$ and $F$ are called base and fiber, respectively. In
particular, if $b=1$, then $B \times _{1}F = B \times F$ is the
usual Cartesian product manifold. For the sake of simplicity, throughout this paper, all relations will be written, without involving the
projection maps from $B \times F $ to each component $B$ and $F$ as in $g=g_{B}\oplus b^{2}g_{F}$.

\begin{proposition} Let $(M,g)$ be an $n$-dimensional pseudo-Riemannian
manifold. Then $(M,g,\varphi,\mu,\lambda)$ is a generalized quasi Yamabe gradient soliton  if
and only if
\begin{equation} \label{meqn2}
{\rm Hess}(\theta)=-\frac{\theta}{m}(\tau-\lambda)g
\end{equation}
where $\mu=1/m$ and $\theta={\rm e}^{-\varphi/m}.$
\end{proposition}

\begin{proof} Introduce $\mu=1/m$ and $\theta={\rm e}^{-\varphi/m}.$ Then
$\nabla \theta = -\frac{\theta}{m} \nabla \varphi $ and
${\rm d}\theta=-\frac{\theta}{m}{\rm d}\varphi.$ So, for vector fields $X$
and $Y$ on $M,$ we have:
\begin{eqnarray*} {\rm Hess}(\varphi)(X,Y) & = & g(\nabla_X \nabla \varphi,Y) \\
& = & g \Bigl(\frac{m}{\theta^2}X(\theta) \nabla \theta-\frac{m}{\theta}\nabla_X \nabla \theta,Y \Bigl) \\
& = & \frac{m}{\theta^2} X(\theta)Y(\theta)-\frac{m}{\theta}{\rm Hess}(\theta) \\
& = & \frac{1}{m} {\rm d} \varphi \otimes {\rm d} \varphi - \frac{m}{\theta} {\rm Hess}(\theta)
\end{eqnarray*}
Thus equation \eqref{meqn1} can be reduced to \eqref{meqn2}.
\end{proof}

As a direct corollary of the above proposition, we have:

\begin{corollary}
Every generalized quasi Yamabe  gradient soliton is conformal gradient soliton.
\end{corollary}

Now, we can refer to the theorem of Cheeger and Colding \cite{Cheeger}, where the authors characterized the warped product manifolds and proved
that any conformal gradient soliton  satisfying  ${\rm Hess}(\theta)=kg$, for some function $k$ is isometric to a warped product on some open
interval. Thus, in view of $\eqref{meqn2}$ we conclude that $(M^n, g)$ is isometric to the warped product $I\times_l N$, for some positive
function $l$, where $I \subseteq \mathbb{R} $ is an
open interval. Thus:
\begin{corollary}
Every generalized quasi Yamabe  gradient soliton admits the warped product structure $I\times_l N$, for some positive function $l$, where $I
\subseteq \mathbb{R} $ is an open interval.
\end{corollary}

This section presents the main result:

\begin{theorem} \label{main-1} Let $M=B \times _{b}F$ be a
warped product manifold equipped with the metric
$g=g_{B}\oplus b^{2}g_{F}.$  Then $(M, g, \varphi, \mu, \lambda)$
is a generalized quasi Yamabe gradient soliton if and only if  the followings hold:
\begin{enumerate}
%\item $(M,g)$ is a conformal gradient soliton,
\item the potential function $\varphi$ depends only on the base manifold $B,$
\item the potential function $\varphi$ and the warping function $b$ cannot be orthogonal,
\item the base manifold $B$ is also a conformal gradient soliton,
\item the scalar curvature $\tau_F$ of the fiber manifold $(F,g_F)$ is constant.
\end{enumerate}
\end{theorem}

\begin{proof} Assume that $(M, g, \varphi, \mu, \lambda)$ is a generalized
quasi Yamabe gradient soliton which is a also warped product. If $X,Y$ are
vector fields on $B$ and $V,W$ are vector fields on $F,$ then apply the last
proposition and obtain: \\
${\rm Hess}(\theta)(X,V)=0$ since $g(X,V)=0.$ On the other hand, by decomposing $\nabla \theta$ on the base and fiber,
\begin{eqnarray*} {\rm Hess}(\theta)(X,V) & = & g(\nabla_X \nabla \theta,V)\\
& = & g(\nabla_X {\rm tan}(\nabla \theta),V) + g(\nabla_X {\rm nor}(\nabla \theta),V) \\
& = & b X(b) g_F({\rm nor}(\nabla \theta),V).
\end{eqnarray*}
Thus, $X(b) g_F({\rm nor}(\nabla \theta),V)=0.$ If $b$ is not constant, the last
equation implies that ${\rm nor}(\nabla \theta)=0.$ So, $\theta={\rm e}^{-\varphi/m}$
depends only on the base manifold $B,$ that is, $\varphi$ is only defined on $B,$
$\varphi \in \mathcal C^\infty(B).$

Moreover, $g(V,W)=b^2 g_F(V,W)$ and $\nabla \theta = {\rm tan}(\nabla \theta)$
since ${\rm nor}(\nabla \theta)=0.$

\begin{eqnarray*} {\rm Hess}(\theta)(V,W) & = & g(\nabla_V \nabla \theta,W) \\
& = & g(\nabla_V {\rm tan}(\nabla \theta),W) \\
& = & b {\rm tan}(\nabla \theta) g_F(V,W).
\end{eqnarray*}
Thus, by the last proposition,
$$b {\rm tan}(\nabla \theta) g_F(V,W)=-\frac{\theta}{m} (\tau-\lambda) b^2 g_F(V,W)$$
Equivalently,
$$ \Bigl( b {\rm tan}(\nabla \theta)(b)+
\frac{b^2 \theta}{m} (\tau-\lambda) \Bigl)g_F(V,W)=0.$$
Contracting the last equation over $V$ and $W,$ we obtain:

$$s b ({\rm tan}\bigl(\nabla \theta)(b) + \frac{b \theta}{m}(\tau-\lambda)\bigl)=0.$$
Noting that ${\rm tan}(\nabla \theta)=\nabla^B(\theta)$ since $\theta \in \mathcal C^\infty(B),$
we have:
$$g_B(\nabla^B \theta, \nabla^B b) = (\lambda-\tau) \frac{b \theta }{m}.$$
Finally,
\begin{eqnarray*} {\rm Hess}(\theta)(X,Y) & = & g(\nabla_X \nabla \theta,Y) \\
& = & g_B(\nabla_X {\rm tan}(\nabla \theta),Y) \\
& = & {\rm Hess}^B(\theta)(X,Y)
\end{eqnarray*}
since $\theta  \in \mathcal C^\infty(B),$ i.e,
${\rm tan}(\nabla \theta)=\nabla^B(\theta).$
By the last proposition, we can have:
$${\rm Hess}^B(\theta)=\frac{\theta}{m}(\lambda-\tau)g_B.$$
Hence, $B$ is conformal gradient soliton and scalar curvature $\tau_F$ of the fiber manifold $(F,g_F)$ is constant.
%the level hypersurfaces of $b$ are totally umbilical (since
%${\rm Hess}^B(\theta)$ is proportional to $g_B$).
The converse statement is straightforward.
\end{proof}

Combining the above results, we can state that:

\begin{corollary}
Let  $(M, g, \varphi, \mu, \lambda)$ be a generalized quasi Yamabe gradient soliton satisfying the conditions (1)-(4) of Theorem \ref{main-1}.
Then it admits the multiply warped product structure $I\times_l N \times_b F$, for some positive functions $l$ and $b$, where $I \subseteq
\mathbb{R} $ is an open interval.
\end{corollary}

\section{Applications}

\subsection{Gradient Yamabe Soliton on Generalized Robertson-Walker Space-times}

We first define generalized Robertson-Walker space-times. Let $(F,g_F)$ be an $%
s-$dimensional Riemannian manifold and $b:I\rightarrow (0,\infty )$ be a
smooth function. Then $(s+1)-$dimensional product manifold $I\times _{b}F$
furnished with the metric tensor
\begin{equation*}
\bar{g}=-\mathrm{d}t^{2}\oplus b^{2}g_F
\end{equation*}%
is called a generalized Robertson-Walker space-time and is denoted by
$M=I \times _{b}F$ where $I$ is an open, connected subinterval of $\mathbb{R}$
and $\mathrm{d}t^{2}$ is the Euclidean metric tensor on $I$. This structure
was introduced to extend Robertson-Walker space-times \cite%
{Sanchez98, Sanchez99} and have been studied by many authors, such
as \cite{ManticaDe1,ManticaDe2,Chen}. From now on, we will denote $\frac{\partial }{%
\partial t}\in \mathfrak{X}(I)$ by $\partial _{t}$ to state our results in
simpler forms.

We will apply our main result Theorem \ref{main-1}. Assume that
$\varphi \in \mathcal C^\infty(I)$ is a potential function for a
generalized Robertson-Walker space-time of the form $M=I \times _{b}F.$

The equation $(\tau-\lambda)g_{ij}={\rm Hess}_{ij}$ yields
\begin{equation}
\label{system.1}
 \left\{ \begin{array}{ll}
\varphi^{\prime \prime} = -(\tau-\lambda),  \\
b^\prime \varphi^\prime = (\tau-\lambda)b.
\end{array} \right.
\end{equation}

Thus $b \varphi^{\prime \prime} = - b^\prime \varphi^\prime.$ By solving
the last ODE, as $b \neq 0,$ we have:
$$\varphi(t)=\alpha \int_{t_0}^t \frac{1}{b(\bar t)}{\rm d}\bar t \quad \text{for some}
\quad \alpha \in \mathbb R.$$

Hence, we can state that:
\begin{theorem}
A generalized Robertson-Walker space-time of the form $M=I \times _{b}F$
is a gradient Yamabe soliton with the potential function $\varphi$ given by
$$\varphi(t)=\alpha \int_{t_0}^t \frac{1}{b(\bar t)}{\rm d}\bar t \quad \text{for some}
\quad \alpha \in \mathbb R.$$

\end{theorem}

\subsection{Gradient Yamabe Soliton on Standard Static Space-times}

We begin by defining standard static space-times. Let $(F,g_F)$ be an
$s-$dimensional Riemannian manifold and $f: F \rightarrow (0,\infty )$ be a smooth
function. Then $(s+1)-$dimensional product manifold $_{f}I\times F$ furnished
with the metric tensor%
\begin{equation*}
g=-f^{2}\mathrm{d}t^{2}\oplus g_F
\end{equation*}%
is called a standard static space-time and is denoted by
$M=_{f} I\times F$ where $I$ is an open, connected subinterval of $\mathbb{R}$
and $dt^{2}$ is the Euclidean metric tensor on $I$.

Note that standard static space-times can be considered as a generalization
of the Einstein static universe\cite{AD1,AD,GES,Besse2008} and many spacetime
models that characterize the universe and the solutions of Einstein's field
equations are known to have this  structure.

Again we apply Theorem \ref{main-1}. Suppose that $\varphi \in
\mathcal C^\infty(F)$ is a potential function for a
standard static space-time of the form $M=_{f} I\times F.$

The equation $(\tau-\lambda)g_{ij}={\rm Hess}_{ij}$ yields
\begin{equation}
\label{system.2}
 \left\{ \begin{array}{ll}
\nabla \varphi(f) = (\tau-\lambda) f,  \\
{\rm Hess}_F(\varphi)=(\tau-\lambda)g_F.
\end{array} \right.
\end{equation}

By contracting the last equation over $F,$ we have $\Delta_F(f)=s(\tau-\lambda).$
Then we obtain: $$\Delta_F(\varphi)=\frac{s}{f}\nabla \varphi(f).$$

Hence, we conclude that:
\begin{theorem}
A standard static space-time of the form  $M=_{f} I\times F$
is a gradient Yamabe soliton with the potential function $\varphi$ given by
$$\Delta_F(\varphi)=\frac{s}{f}\nabla \varphi(f).$$

\end{theorem}

\subsection{Gradient Yamabe Soliton on 3-dimensional Walker Manifolds}

Three-dimensional Lorentzian manifolds admitting a parallel degenerate line field is called Walker manifold, \cite{Walker1,Walker2}. Then then
there exist local coordinates $(t,x,y)$ such that with respect to the local frame fields $\{ \partial_t, \partial_x, \partial_y \}$ the
Lorentzian metric tensor is of the form:
\begin{equation}
\label{Walker metric}
g=2{\rm d}t{\rm d}y + {\rm d}x^2 + \phi(t,x,y) {\rm d}y^2,
\end{equation}
for some function $ \phi(t,x,y)$. The restricted case of Walker manifold where $\phi$ described as a function of only $x$ and $y$ is called
strictly Walker manifold and these classes of Walker manifolds are geodesically complete. Also, it is known that a Walker manifold is Einstein
if and only if it is flat, \cite{Walker1}.

Now, we will investigate conditions on this particular class of manifolds
to have gradient Yamabe solitons, that is,
\begin{equation} \label{maineqn3WM}{\rm Hess}(f)_{ij} = (\tau-\lambda)g_{ij}
\end{equation}
where $f$ is a potential function.

Note that this equation implies that
$$\frac{\Delta(f)}{3}=\tau-\lambda.$$

By using the metric $\eqref{Walker metric}$ and straightforward computations,
we have:
\begin{equation}
\label{system.3}
 \left\{ \begin{array}{ll}
{\rm Hess}(f)_{tt} = f_{tt}, \\
{\rm Hess}(f)_{tx} = f_{tx}, \\
{\rm Hess}(f)_{ty} = f_{ty} - \frac{1}{2} \phi_t f_t, \\
{\rm Hess}(f)_{xx} = f_{xx}, \\
{\rm Hess}(f)_{xy} = f_{xy} - \frac{1}{2} \phi_x f_t, \\
{\rm Hess}(f)_{yy} = f_{yy} - \frac{1}{2} \phi \phi_t f_t + \frac{1}{2} \phi_t f_y.
\end{array} \right.
\end{equation}

Moreover,
\begin{equation} \label{f_delta}
\Delta(f)=-\phi f_{tt} + 2f_{ty} - \phi_t f_t + f_{xx}
\end{equation}

By combining these, we get:

\begin{equation} \label{eqn3WM}
\tau-\lambda= \frac{1}{3} \Bigl( -\phi f_{tt} +2 f_{ty} - \phi_t f_t + f_{xx} \Bigl)
\end{equation}

By applying equations \eqref{maineqn3WM}, $\eqref{system.3}$ and \eqref{eqn3WM}, we obtain the following system of PDEs:
\begin{equation}
\label{system.4}
 \left\{ \begin{array}{ll}
f_{tt}=0, \\
f_{tx}=0,\\
f_{xy}=\frac{1}{2} \phi_x f_t,\\
f_{xx}-f_{ty}=-\frac{1}{2} \phi_t f_t,\\
f_{yy}-f_{xx}=\frac{1}{2}\phi \phi_t f_t - \frac{1}{2} \phi_t f_y.
\end{array} \right.
\end{equation}

Notice first that $f_{tt}=0$ and $f_{tx}=0$ imply that $f(t,x,y)=t b(y) + c(x,y)$
for some functions $b$ and $c.$

Thus $f_{xy}=\frac{1}{2} \phi_x f_t$ turns out to be $2c_{xy}(x,y)=\phi_x b(y).$
Substituting them in the last two equations of the system $\eqref{system.4}$, we obtain

\begin{equation}
\label{wm_m1}
2c_{xx}(x,y) + \phi_t b(y) = 2b^\prime(y)
\end{equation}

Differentiating the both sides of equation \eqref{wm_m1} with respect to
$t,$ we get $\phi_{tt}b(y)=0$ this implies that $\phi_{tt}=0$ or $b(y)=0.$

$\bullet$ \textit{Case 1}: $\phi_{tt}=0.$ Then  the metric becomes Ricci-flat so we may assume that this case should be excluded.

$\bullet$ \textit{Case 2}: If $\phi_{tt} \neq 0,$ then $b(y)=0,$ i.e, $f(t,x,y)=c(x,y).$

Thus $c_{xx}=0$ and $c_{xy}=0$ imply that $c(x,y)=\kappa x+\eta(y)$, for some $\kappa \in \mathbb{R}$.  Hence
$2 \eta^{\prime \prime}(y) + \phi_t \eta^\prime(y)=0.$ So, $\phi_t=-2\ln[\eta^\prime(y)]$
and then $\phi(t,x,y)=-2t\ln[\eta^\prime(y)] + \zeta(x,y).$ In this case, the potential
function is given by: $f(x,y)=\kappa x+ \eta(y).$ Hence we obtain the main result of this section:

\begin{theorem} \label{thm:Walker}
Let $(M,g)$ be a 3-dimensional Lorenztian Walker manifold
equipped with metric:
$$g=2{\rm d}t{\rm d}y + {\rm d}x^2 + \phi(t,x,y) {\rm d}y^2.$$
Then $(M,g)$ is a gradient Yamabe soliton if and only if
\begin{enumerate}
\item the potential function of the soliton structure is given by \\
$f(x,y)=\kappa x+ \eta(y)$ for some $\kappa  \in \mathbb R.$
\item the function defined in the Walker metric is given by \\
$\phi(t,x,y)=-2t\ln[\eta^\prime(y)] + \zeta(x,y).$
\end{enumerate}
\end{theorem}

Note that by Equation \eqref{f_delta}, we can get: $\Delta(f)=0.$ In order to have a non-trivial potential function we may assume that $M$ is
non-compact due to Hopf's Lemma. Thus, we have:
\begin{corollary}
Let $(M,g)$ be a 3-dimensional Lorenztian Walker manifold equipped with metric $\eqref{Walker metric}$. Then the gradient Yamabe soliton
structure on $(M,g)$ characterized by the conditions (1) and (2) of Theorem $\ref{thm:Walker}$  is non-compact.
\end{corollary}

\subsection{Gradient Yamabe Soliton on 4-dimensional Walker Manifolds}

Before finishing the work, let's finally construct the gradient Yamabe soliton structure on the 4-dimensional Walker manifold.

A 4-dimensional Walker manifold is a triple $(M,g,D)$ consisting of  an indefinite metric $g$ and an 2-dimensional parallel null plane $D$ and
in this case $g$ has neutral signature $(-,-,+,+)$ and in suitable coordinates $(x,y,z,t)$ such that with respect to the local frame fields $\{
\partial_x, \partial_y,\partial_z, \partial_t \}$  it can be given by \cite{Walker1}
\begin{equation}
\label{4-Walker metric}
g=2{\rm d}x{\rm d}z +2{\rm d}y{\rm d}t + a(x,y,z,t) {\rm d}z^2 + 2c(x,y,z,t){\rm d}z{\rm d}t +b(x,y,z,t) {\rm d}t^2,
\end{equation}
for some functions $ a(x,y,z,t), b(x,y,z,t), c(x,y,z,t)$ and $D=<\partial_x, \partial_y>$.
The case where $c(x, y, z, t) =0$ was studied  and locally conformally flatness of this special metric was investigated in \cite{Walker3}.

Here, we consider the restricted case of Walker manifold where $a=c=0$ and $b$ described as a function of only $t$ so the metric
$\eqref{4-Walker metric}$ reduces to the form
\begin{equation}
\label{4-Walker metric-b}
g=2{\rm d}x{\rm d}z +2{\rm d}y{\rm d}t +b(t) {\rm d}t^2,
\end{equation}

Now, we will investigate conditions on this particular class of manifolds
to have gradient Yamabe solitons, that is,
\begin{equation} \label{maineqn4WM}{\rm Hess}(f)_{ij} = (\tau-\lambda)g_{ij}
\end{equation}
where $f$ is a potential function.

Note that this equation implies that
$$\frac{\Delta(f)}{4}=\tau-\lambda.$$

By using the metric $\eqref{4-Walker metric-b}$ and straightforward computations,
we have:
\begin{equation}
\label{system.5}
 \left\{ \begin{array}{ll}
{\rm Hess}(f)_{xx} = f_{xx}, \ \
{\rm Hess}(f)_{xy} = f_{xy}, \\
{\rm Hess}(f)_{xz} = f_{xz}, \ \
{\rm Hess}(f)_{xt} = f_{xt},\\
{\rm Hess}(f)_{yy} = f_{yy}, \ \
{\rm Hess}(f)_{yz} = f_{yz}, \\
{\rm Hess}(f)_{yt} = f_{yt}, \ \
{\rm Hess}(f)_{zz} = f_{zz}, \\
{\rm Hess}(f)_{zt} = f_{zt}, \ \
{\rm Hess}(f)_{tt} = f_{tt} -\frac{1}{2}b_tf_y.
\end{array} \right.
\end{equation}

Moreover,
\begin{equation} \label{f_delta-4}
\Delta(f)=2f_{xz}-bf_{yy}+2f_{yt}.
\end{equation}

By combining these, we get:

\begin{equation} \label{eqn4WM}
\tau-\lambda= \frac{1}{4} \Bigl( 2f_{xz}-bf_{yy}+2f_{yt} \Bigl)
\end{equation}

By applying equations \eqref{maineqn4WM}, $\eqref{system.5}$ and \eqref{eqn4WM}, we obtain the following system of PDEs:
\begin{equation}
\label{system.6}
 \left\{ \begin{array}{ll}
f_{xx}=f_{xy}=f_{yy}=f_{yz}=f_{zz}=0, \\
f_{xt}=f_{zt}=0,\\
f_{xz}=f_{yt}=\frac{\Delta (f)}{4}, \\
f_{tt}-\frac{1}{2}b_tf_y=b\frac{\Delta (f)}{4}.
\end{array} \right.
\end{equation}

Notice first that $f_{xy}=0$, $f_{xx}=0$, $f_{yz}=0$ and $f_{yy}=0$  imply that $f(x,y,z,t)=x \beta (z,t) + y A(t) +B(z,t)$ for some functions
$\beta, A $ and $B.$
Thus from $f_{xz}=f_{yt}$, we have $\beta_z(z,t)=A'(t)$.  Since $f_{zz}=0$ and $f_{xt}=0$, $A'(t)=C'(z)$ which yields
\begin{equation}
A(t)=c_0 t+c_1, \ \ C(z)=c_0 z+c_2, \ \ \textrm{where} \ c_0,c_1,c_2 \in \mathbb{R}.
\end{equation}
Also, from  $f_{xt}=0$ and $f_{zt}=0$, we obtain $\beta (y,z)=\kappa z +\mu$ and  $B(z,t)=c_3z+E(t)$, respectively where $c_0 \in \mathbb{R}$.
As a  result, the potential function is given by
\begin{equation}
f(x,y,z,t)=x (c_0 z +c_2) + y (c_0 z +c_1) +c_3z+E(t).
\end{equation}
Substituting this in the last equation of the system $\eqref{system.6}$, we obtain

\begin{equation}
\label{wm_m2}
2E''(t)-b_t(c_0 t+c_1) =2c_0 b.
\end{equation}
Integrating \eqref{wm_m2} with respect to
$t$ and using integration by parts for the second term,    we get
$2E'(t)-b(c_0 t+c_1)=c_0 \int_{t_0}^t b(\bar t){\rm d}\bar t $.

 Hence we obtain the main result of this section:

\begin{theorem} \label{thm:4Walker}
Let $(M,g)$ be a 4-dimensional  Walker manifold
equipped with metric:
\begin{equation}
\label{4-Walker metric-b(t)}
g=2{\rm d}x{\rm d}z +2{\rm d}y{\rm d}t +b(t) {\rm d}t^2.
\end{equation}
Then $(M,g)$ is a gradient Yamabe soliton if and only if its potential function is defined by
\begin{equation*}
f(x,y,z,t)=x (c_0 z +c_2) + y (c_0 z +c_1) +c_3z+E(t)
\end{equation*}
where $c_i \in \mathbb{R}$, ($i=0,1,2,3$) and the metric function $b(t) $ and $E(t) $ are related by
\begin{equation*}
2E'(t)-b(c_0 t+c_1)=c_0 \int_{t_0}^t b(\bar t){\rm d}\bar t.
\end{equation*}
Moreover, $\Delta (f)=4c_0$.
\end{theorem}

If in additionally, $M$ is compact, then by  applying the Divergence Theorem, $c_0=0$ and so get conclude that:

\begin{corollary}
Let $(M,g)$ be a 4-dimensional  Walker manifold equipped with metric \eqref{4-Walker metric-b(t)}.
Then $(M,g)$ is a compact gradient Yamabe soliton if and only if its potential function is defined by
\begin{equation*}
f(x,y,z,t)=c_2x + c_1y +c_3z+\frac{c_1}{2} \int_{t_0}^t b(\bar t){\rm d}\bar t.
\end{equation*}
where $c_i \in \mathbb{R}$, ($i=1,2,3$).
\end{corollary}

\end{document}